\documentclass[12pt,a4paper]{article}
\usepackage{amsmath}
\usepackage{amssymb}
\usepackage{amsxtra}
\usepackage{latexsym}
\usepackage{mathtext}
\usepackage[cp1251]{inputenc}
\usepackage[russian,english,ukrainian]{babel}
\usepackage{graphics}

\newtheorem{thm}{\rm\bf Theorem}[section]
\newtheorem{lem}[thm]{\rm\bf Lemma}
\newtheorem{cor}[thm]{\rm\bf Corollary}
\newtheorem{rem}[thm]{\rm\bf Remark}
\newtheorem{df}[thm]{\rm\bf Definition}

\newenvironment{proof}{\noindent{\bf Proof.}}{$\blacksquare$}

\begin{document}

\selectlanguage{english}

\noindent {\it UDK 517.929}

\begin{center}
{\bf \Large On exact controllability of neutral type time-delay systems}

\vskip4mm

\foreignlanguage{russian}{
{ \bf \Large К вопросу точной управляемости систем с запаздыванием нейтрального типа}
}
\end{center}

\noindent
R.~Rabah, IRCCyN, \'Ecole des Mines de Nantes, 4 rue Alfred Kastler, BP 20722 Nantes, France.
E-mail: rabah.rabah@mines-nantes.fr.

\vskip2mm
\noindent
G.~M.~Sklyar, Institute of Mathematics, University of Szczecin and Karazin Kharkiv University,
Wielkopolska 15, 70-451, Szczecin, Poland. E-mail:  sklar@univ.szczecin.pl.

\vskip2mm
\noindent
P.~Yu.~Barkhayev, Institute for Low Temperature Physics and Engineering,
Academy of Sciences of Ukraine and Karazin Kharkiv University, 47 Lenin Ave.,  61103 Kharkiv, Ukraine.
E-mail:   barkhayev@ilt.kharkov.ua.

\begin{abstract}
The paper is devoted to the problem of global exact controllability for a wide class of
neutral and mixed time-delay systems.
We consider an equivalent operator model in Hilbert space and
formulate steering conditions of controllable states as a vector moment problem.
The existence of a basis of eigenvectors of system operator allows to simplify substantially the form of
the moment problem. A change of control by a feedback law allows to modify the  system structure
to guarantee   the existence of
a basis of eigenvectors of the corresponding operator.
We prove a criterion of exact controllability and ascertain the precise critical time of controllability.
\end{abstract}

\begin{otherlanguage}{ukrainian}
\begin{abstract}
Робота присвячена вирішенню задачі глобальної точної керованості
для досить широкого класу систем з запізненням нейтрального та змішаного типів.
Розглядаючи еквівалентну операторну модель в гільбертовому просторі,
ми формулюємо умови керованості у вигляді деякої векторної проблеми моментів.
Вид даної проблеми моментів істотно спрощується при наявності базису простору з власних
векторів оператора системи з запізненням.
Заміна керування дозволяє перетворити структуру системи, і гарантувати існування базису з власних векторів
відповідного оператора.
Ми доводимо критерій точної керованості і встановлюємо точний час керування.
\end{abstract}
\end{otherlanguage}

\foreignlanguage{russian}{
\begin{abstract}
Данная работа посвящена решению задачи глобальной точной управляемости
для достаточно широкого класса систем с запаздыванием нейтрального и смешанного типов.
Рассматривая эквивалентную операторную модель в гильбертовом пространстве,
мы формулируем условия управляемости в виде некоторой векторной проблемы моментов.
Вид данной проблемы моментов существенно упрощается при наличии базиса пространства из собственных
векторов оператора системы с запаздыванием.
Замена управления позволяет преобразовать структуру системы, и гарантировать существование базиса из собственных векторов
соответствующего оператора.
Мы доказываем критерий точной управляемости и устанавливаем точное время управления.
\end{abstract}
}

\newpage

\section{Introduction}
The controllability problem for linear time delay systems has quite a long history
(see, e.g. \cite{Bensoussan_et_al_1992,Banks_Jacobs_Langenhop_1975,Gabasov-kirillova_1971,
Jacobs_Langenhop_1976,Manitius_Triggiani_1978,Marchenko_1979}
and references therein).
In this paper we consider the problem of global exact controllability for
a large class of neutral type systems given by the following equation:
\begin{equation}\label{ctr_eqn_1s}
\displaystyle
\dot{z}(t) = A_{-1}\dot{z}(t-1)+L z_t+Bu,\quad t\ge 0,
\end{equation}
where $A_{-1}\in\mathbb{R}^{n\times n}$, $B\in\mathbb{R}^{n\times r}$ are constant matrices,
$z_t: [-1, 0]\rightarrow \mathbb{C}^n$ is the history of $z$ defined by $z_t(s)=z(t+s)$,
the delay operator $L$ is given by
$$L f=\int^0_{-1}A_2(\theta)\frac{{\;{\rm d}}}{{\;{\rm d}} \theta}f(\theta)\,
{\;{\rm d}}\theta+\int^0_{-1}A_3(\theta)f(\theta)\,{\;{\rm d}}\theta,$$
and $A_2$, $A_3$ are $n\times n$-matrices whose elements belong to $L_2([-1,0],\mathbb{C})$.

The representation of delay systems as systems in some functional space proved to be one of the most productive approaches.
Namely, it is possible to associate with delay system the following infinite-dimensional model:
\begin{equation}\label{ctr_eqn_1s0}
\dot{x}={\mathcal A}x+{\mathcal B} u,\quad x\in H,
\end{equation}
where $H$ is a Hilbert space and the linear operator ${\mathcal A}$ is the generator of a $C_0$-semigroup.

For finite-dimensional linear control systems of the form~(\ref{ctr_eqn_1s0}),
Kalman's controllability concept is well-known: the reachability set from $0$ at time $T$ coincides
with the whole phase space (${\mathcal R}_T=H$) for some $T>0$. Moreover, if there are no constraints on control,
then controllability time $T$ may be chosen arbitrarily.
However, if the  phase space $H$ is infinite-dimensional, then the described property does not hold, in general.
For delay systems, the  reachability set is always a subset of the domain ${\mathcal D}({\mathcal A})$
of the operator ${\mathcal A}$, thus, it is natural to pose the problem of reaching the whole set ${\mathcal D}({\mathcal A})$.
Besides, for delay systems the minimal controllability time can not be arbitrarily small,
what leads us to the problem of finding this minimal time of transfer from
$0$ to an arbitrary state of ${\mathcal D}({\mathcal A})$.
The following criterion of   exact  controllability had been obtained by coauthors
of the present paper \cite{Rabah_Sklyar_2007}.

\begin{thm}\label{ctr_thr_intr}
Neutral type system (\ref{ctr_eqn_1s}) is exactly  controllable if and only if the
following conditions are verified:
\begin{enumerate}
 \item[{\rm (i)}] there are no $\lambda \in \mathbb{C}$ and $y\in\mathbb{C}^n\backslash\{0\}$,
 such that $\left(\Delta_{\mathcal A}(\lambda)\right)^*y=0$ and $B^*y=0$, where
\begin{equation} \label{ctr_eqn_Delta}
\displaystyle
\Delta_{\mathcal A}(\lambda) = \lambda I
- \lambda {\rm e}^{-\lambda}A_{-1}-
 \lambda \int\nolimits^0_{-1}{\rm e}^{\lambda s} A_2(s){\mathrm d}s  - \int\nolimits^0_{-1}{\rm e}^{\lambda s} A_3(s){\mathrm d}s,
\end{equation}
or equivalently, ${\rm rank} (\Delta_{\mathcal A} (\lambda)\;  B)=n$ for any $\lambda \in \mathbb{C}$.
\item[{\rm (ii)}] there are no $\mu \in \sigma(A_{-1})$ and $y\in\mathbb{C}^n\backslash\{0\}$, such that
$A_{-1}^*y=\bar \mu y$ and $ B^*y=0$, or equivalently,
${\rm rank}(B\quad A_{-1}B\;\cdots\; A_{-1}^{n-1}B)=n$.
\end{enumerate}
Moreover, if the conditions {\rm (i)} and {\rm (ii)} hold, then the system is exactly controllable
at any time $T>n_1$ and not  exactly  controllable at any time $T\le n_1$,
where $n_1$ is the first controllability index of the pair $(A_{-1},B)$.

If maximal delay is equal to $h$  then the critical time of controllability equals to $T=n_1h$.
 \end{thm}

We note, that (\ref{ctr_eqn_1s}) is a system with distributed delay for which,
in contrast to systems with several discrete delays
(see \cite{Banks_Jacobs_Langenhop_1975,Jacobs_Langenhop_1976,Bartosiewicz_1983,
Khartovskii_Pavlovskaya_2013,OConnor_Tarn_1983b,Rivera_Langenhop_1978} and references therein),
 the explicit form of the semigroup is unknown, in general,
what makes the analysis much more complicated.
We also note, that an important advantage of the theorem is
to give the exact critical time of controllability.

Besides, we can note that for linear retarded systems ($A_{-1}=0$),
the conditions of exact controllability imply $\mathrm{rank\,} B=n$,
which is a very strong condition, this means that exact controllability is more typical for neutral type systems.

To study the exact controllability we use the moment problem approach: the steering conditions of controllable states are
represented as a vectorial trigonometric moment problem with
respect to a special Riesz basis. We analyze the solvability
of the obtained non-Fourier moment problem using methods developed in \cite{Avdonin_Ivanov_1995}
(see also \cite{Young_1980}).

The existence of a basis of the state space consisting of eigenvectors (or generalized eigenvectors)
simplifies essentially the expression of the moment problem (see \cite{Rabah_Sklyar_Springer} and \cite{Shklyar_2011}).
 In our case, the existence of a basis of eigenvectors  is determined by
 the form of the matrix $A_{-1}$ of neutral term of the system (\ref{ctr_eqn_1s}), and, in general,
such basis does not exist (see \cite{Rabah_Sklyar_Rezounenko_2003,Rabah_Sklyar_Rezounenko_2005}).
This makes quite sophisticated the procedure of the choice of a Riesz basis and further manipulations with it
in general case (\cite{Rabah_Sklyar_2007}).

However, by means of a change of control in the initial  system, it is
possible to pass over to an equivalent controllability problem for a system with a matrix $A_{-1}$ of a simple structure.
This structure guarantees the existence of a Riesz basis of
 eigenvectors for the state space.
The form of the corresponding moment problem becomes  simpler
what makes the constructions and the proofs of the main results clear and  illustrative.

In this paper we give the proof of Theorem~\ref{ctr_thr_intr} for the system (\ref{ctr_eqn_1s}) with $A_{-1}$ of a special form
and show that this fact implies the proof for a system with an arbitrary matrix $A_{-1}$.
Besides, we consider the controllability problem for co-called mixed retarded-neutral type systems
(see also \cite{Rabah_Sklyar_Barkhayev_2012}),
which was considered in \cite{Rabah_Sklyar_2007}, and prove that if the neutral term is singular ($\det A_{-1}=0$)
and the pair $(A_{-1}, B)$ is uncontrollable, then the system (\ref{ctr_eqn_1s}) is uncontrollable as well.

The paper is organized as follows.  In
Section~\ref{sec:prelim} we introduce the abstract equation and discuss how we can consider without loss of generality that the system
has a special form with a Riesz basis of eigenvectors.
In Section~\ref{sec:Rieszbasis}, using spectral Riesz bases, we represent the steering conditions as a vectorial moment problem.
Section~\ref{sec:necess} is devoted to proof of necessity of controllability conditions and in Sections~\ref{sec:suff1} and~\ref{sec:suff2} we prove
sufficiency of these conditions for the cases of one-dimensional and multi-dimensional controls.
Finally, in Section~\ref{sec:example} we give an example illustrating the obtained results.

\section{Equivalent systems}\label{sec:prelim}%
We consider the operator model of time-delay systems introduced in \cite{Burns_Herdman_Stech_1983} (see also \cite{Ito_Tarn_1985}).
The state space is $M_2(-1,0; \mathbb{C}^n)=\mathbb{C}^n \times L_2(-1,0; \mathbb{C}^n)$, shortly $M_2$,
and the problem (\ref{ctr_eqn_1s}) may be rewritten in the following form:
\begin{equation}\label{ctr_eqn_*}
\dot x(t) = {\mathcal {A}} x(t) + {\mathcal {B}} u(t), \quad
{\mathcal {A}} =\left(\begin{array}{cc} 0 & L \\ 0 & \frac{{\;{\rm d}}}{{\;{\rm d}} \theta} \end{array}\right),\;
{\mathcal {B}} =\left(\begin{array}{c} B \\ 0\end{array}\right),
\end{equation}
where the domain of the operator $\mathcal A$ is
$${\mathcal D}({\mathcal A})= \{ \left (y,  z(\cdot)\right ) \in M_2 : z\in H^1(-1,0;\mathbb{C}^n), y=z(0)-A_{-1}z(-1)\}.$$

The reachability set from the initial state 0 at time T is defined by
$$
{\mathcal {R}}_T=\left \{x: x=\int\nolimits_0^T {\rm e}^{{\mathcal {A}} t}{\mathcal {B}} u(t){\;{\rm d}} t,
\quad u(\cdot)\in L_2(0,T;\mathbb{C}^r)\right\}.
$$
Further we show that ${\mathcal {R}}_T\subset{\mathcal {D}}(\mathcal {A})$ for all $T>0$.
\begin{df}
We say that the system {\rm (\ref{ctr_eqn_*})} is exactly controllable from zero by controls from $L_2$,
if there exists a time $T_0$ (critical time), such that for all $T>T_0$ one has
$${\mathcal {R}}_T={\mathcal {D}}({\mathcal {A}}),$$
and for all $T<T_0$: ${\mathcal {R}}_T\not={\mathcal {D}}({\mathcal {A}})$.
\end{df}
The given definition means that for some $T>0$ the set of solutions $\{z(t):\: t\in[T-1,T]\}$
of the system~(\ref{ctr_eqn_1s}) coincides with space $H^1(T-1,T;\mathbb{C}^n)$.

\begin{lem}\label{ctr_lem_lem7.3}
If the system~(\ref{ctr_eqn_1s}) is exactly  controllable at time $T$, then
for any matrix $P\in\mathbb{C}^{n\times r}$ the perturbed system
\begin{eqnarray}\label{ctr_eqn_1sp}
\displaystyle
\dot z(t) =  (A_{-1}+BP)\dot z(t-1) + L z_t + Bu
\end{eqnarray}
is exactly controllable at the same time $T$.
\end{lem}
\begin{proof}
Assume that the system(\ref{ctr_eqn_1s}) is controllable at the time $T$.
This means that for any function $f(t) \in H^1(T-1,T;\mathbb{C}^n)$
there exists a control $u(t) \in L_2(0,T;\mathbb{C}^n)$, such that the solution of the equation
\begin{equation}\label{ctr_eqn_1st}
\displaystyle
\dot z(t) = A_{-1}\dot z(t-1) + L z_t + Bu(t),
\end{equation}
with the initial condition $z(t)=0$, $t\in[-1,0]$ satisfies the relation $z(t)=f(t)$, $t\in[T-1,T]$.
Let us rewrite (\ref{ctr_eqn_1st}) in the form
$$
\displaystyle
\dot z(t) = (A_{-1}+BP)\dot z(t-1) + L z_t + Bv(t),
$$
where $v(t)=u(t)-P\dot z(t-1)$, $t \in [0,T]$. Since $z(t-1) \in H^1(0,T;\mathbb{C}^n)$,
then $v(t) \in L_2(0,T;\mathbb{C}^n)$.
Therefore, the control $v(t)$ transfers the state $z(t)=0$, $t\in [-1,0]$
to the state $z(t)=f(t)$, $t\in [T-1,T]$ by virtue of the system (\ref{ctr_eqn_1sp}).
This means that (\ref{ctr_eqn_1sp}) is also exactly controllable at the time $T$.
\end{proof}

 We have also an equivalence in the conditions of exact controllabilty in Theorem~\ref{ctr_thr_intr}.
\begin{lem} If system~(\ref{ctr_eqn_1s}) satisfies the conditions (i) and (ii) of Theorem~\ref{ctr_thr_intr},
then a perturbed system~(\ref{ctr_eqn_1sp}) with an arbitrary matrix $P$ satisfies the same conditions.
\end{lem}
\begin{proof}
Indeed, let us denote the operator corresponding to the system~(\ref{ctr_eqn_1sp}) by $\widehat{\mathcal A}$.
Thus, if the condition (i) does not hold for (\ref{ctr_eqn_1s}): $\Delta_{\mathcal A}^*(\lambda)y=0$ and $B^*y=0$,
then
$$
\Delta_{\widehat{\mathcal A}}^*(\lambda)y=[\Delta_{\mathcal A}^*(\lambda)-\lambda {\rm e}^{-\lambda}P^*B^*]y=0
$$
what means that the condition (i) does not hold for the system~(\ref{ctr_eqn_1sp}).

The equivalency of the condition~(ii) for systems (\ref{ctr_eqn_1s}) and (\ref{ctr_eqn_1sp})
is a well-known classical result (see, e.g. \cite{Wonham_1985}):
$$
{\mathrm {rank}\,}(B\quad A_{-1}B\;\cdots\; A_{-1}^{n-1}B)= {\mathrm {rank}\,}(B\quad (A_{-1}+BP)B\;\cdots\; (A_{-1}+BP)^{n-1}B).
$$
\end{proof}
\begin{cor}
Therefore, if we prove Theorem~\ref{ctr_thr_intr} for system~(\ref{ctr_eqn_1s}) with a pair $(A_{-1},B)$,
then, we also prove this theorem for all systems with pair of matrices $(\widehat{A}_{-1}, B)$,
where $\widehat{A}_{-1}=A_{-1}+BP$.
\end{cor}
If the pair $(A,B)$ is controllable, then (see, e.g. \cite{Wonham_1985})
for any set $S=\{\mu_1,\ldots,\mu_n\}\subset\mathbb{C}$,
there exists matrix $P\in\mathbb{C}^{r\times n}$ such that the set $S$
is the spectrum of $\sigma(A+BP)=S$.
Thus, if we fix $n$ distinct real numbers
\begin{equation}\label{ctr_eqn_spec_A-1}
\{\mu_1,\ldots,\mu_n\}\subset\mathbb{R}, \quad\mu_i\not=\mu_j, i\not=j,\quad \mu_i\not\in\{0,1\},
\end{equation}
we can find  a change of control $u(t)=P\dot{z}(t-1)+v(t)$, $P\in\mathbb{C}^{r\times n}$,
and a tranformation of the state $z=Cw$, which reduce the system to the following form
\begin{equation}\label{ctr_eqn_1s_11}
\dot{w}(t) = \widehat{A}_{-1}\dot{w}(t-1)+
\int^0_{-1}\widehat{A}_2(\theta) \dot{w}(t+\theta) {\;{\rm d}}\theta
+\int^0_{-1}\widehat{A}_3(\theta){w}(t+\theta){\;{\rm d}}\theta +\widehat{B}v,
\end{equation}
where $\widehat{A}_{-1}=C^{-1}(A_{-1}+BP)C$, $\widehat{A}_i(\theta)=C^{-1}{A}_i(\theta)C$, $\widehat{B}=C^{-1}B$,
satisfy the following conditions:
\begin{itemize}
\item[(a)] the spectrum of $\widehat{A}_{-1}$ is $\sigma(\widehat{A}_{-1})=\{\mu_m\}_{m=1}^n$;
\item[(b)] the pair $(\widehat{A}_{-1}, \widehat{B})$ is in Frobenius normal form (see \cite{Wonham_1985}), i.e.
\begin{equation}\label{ctr_eqn_frobenius}
\widehat{A}_{-1}={\rm diag}\{F_1,\dots,F_r\},\; F_i=\left(
\begin{array}{ccccc}
0&1&0&\cdots &0 \\
0&0&1&\cdots &0 \\
\vdots & \vdots & \vdots &\ddots &\vdots \\
0&0&0&\cdots&1\\
a_1^i&a_2^i&a_3^i&\cdots&a_{s_i}^i
\end{array}
\right)
\end{equation}
and $\widehat{B}={\rm diag}\{g_1,\dots,g_r\}$, where $g_i=(0,\:0,\ldots, 1)^{\mathrm T}\in \mathbb{C}^{s_i}$.
\end{itemize}
From these considerations we obtain the following lemma.
\begin{lem}
 The proof of sufficiency of Theorem~\ref{ctr_thr_intr} for the familly of systems~(\ref{ctr_eqn_1s_11}), verifying conditions
(a)-(b), implies the sufficiency of these condition for arbitrary systems of type~(\ref{ctr_eqn_1s}).
\end{lem}
\begin{rem}
In the proof of Theorem~\ref{ctr_thr_intr} for the case of one-dimensional control ($r=1$) it is enough
to assume only the condition (a). However, in the proof of general the  case (multidimensional control)
we need both conditions (a) and (b).
\end{rem}

In the paper \cite {Rabah_Sklyar_2007} the necessity of condition (ii) is proved with the assumption that
the matrix  $A_{-1}$   is non-singular. In the present paper, we complete the proof: if the pair  $(A_{-1},B)$ is not controllable,
then the system ~(\ref{ctr_eqn_1s}) is not controllable as
well (Theorem~\ref{ctr_thr_thm63}).

Further, without loss of generality, we may assume that the conditions (a) and (b)
hold for the pair $(A_{-1}, B)$.
Due to this construction, we have  $\det A_{-1}\not=0$ and we denote by $\{c_m\}_{m=1}^n$
the basis of normed eigenvectors of $A_{-1}$.
\section{Riesz basis and the moment problem}\label{sec:Rieszbasis}

Let us denote by $\widetilde {\mathcal A}$ the operator $\mathcal A$ in the case $A_2(\theta)=A_3(\theta)\equiv 0$.
The eigenvalues of $\widetilde {\mathcal A}$ are of the form (see \cite{Rabah_Sklyar_Rezounenko_2005}):
$$
\sigma({\widetilde {\mathcal A})=\{ \widetilde{\lambda}_m^{k}= \ln|\mu_m|}
+ 2k\pi {\mathrm  i}, \; m=1,\ldots,n,\: k\in\mathbb{Z}\} \cup \{0\},
$$
where $\{\mu_1,\ldots, \mu_n\}= \sigma(A_{-1})$.
Since all eigenvalues of $A_{-1}$ are simple, then
the operator $\widetilde{\mathcal A}$ possesses simple eigenvalues only,
and to each eigenvalue $\widetilde{\lambda}_m^{k}$
corresponds only one eigenvector $\widetilde{\varphi}_{m,k}=\left(0,{\rm e}^{\widetilde{\lambda}_m^{k}t}c_m \right)^T$
and there are no root-vectors.
Moreover, the following estimates hold
$$
0<\inf\limits_{k\in \mathbb{Z}}\|\widetilde \varphi_{m,k}\|
\le \sup\limits_{k\in \mathbb{Z}}\|\widetilde \varphi_{m,k}\| <+\infty.
$$

The spectrum of $\mathcal A$ is of the following form (see \cite{Rabah_Sklyar_Rezounenko_2005}):
$$
\sigma(\mathcal A)=\{\ln|\mu_m| + 2k\pi {\mathrm  i} + O(1/k), \; m=1,\ldots,n, k\in\mathbb{Z}\}.
$$
There exists $N\in\mathbb{N}$ such that for all $m=1,\ldots,n$ and for all $k: |k|> N$
the total multiplicity of the eigenvalues of $\mathcal A$, contained in the circles $L_m^{k}(r^{(k)})$
equals to 1, where $L_m^{k}(r^{(k)})=L_m^{k}$ are circles with radii $r^{(k)}$  centered at $\widetilde{\lambda}_m^{k}$,
and the relation $\sum\limits_{k\in\mathbb{Z}}(r^{(k)})^2 < \infty$ is satisfied (\cite[Theorem~4]{Rabah_Sklyar_Rezounenko_2008}).
We denote these eigenvalues of the operator ${\mathcal A}$ as $\lambda_m^k$ and the corresponding
eigenvectors as $\varphi_{m,k}$, $m=1,\ldots,n$, $|k|> N$.

Assume that the vectors $\varphi_{m,k}$ are normed such that
$P_m^{(k)}\widetilde{\varphi}_{m,k}= \varphi_{m,k}$, where
$P_m^{(k)}=\frac{1}{2\pi {\mathrm i}}\int_{L^{(k)}_m} R(\lambda,{\mathcal {A}}) {\mathrm d}\lambda.$
The families $\{\varphi_{m,k}\}$ and $\{\widetilde{\varphi}_{m,k}\}$ are quadratically close:
$
\sum\limits_{|k|>N}\sum\limits_{m=1}^n \|\varphi_{m,k} - \widetilde{\varphi}_{m,k}\|^2<\infty
$,
what, in particular, implies the following estimates
\begin{equation}\label{ctr_eqn_1.2}
0<\inf_{|k|>N}\|\varphi_{m,k}\|\le \sup_{|k|>N}\|\varphi_{m,k}\| <+ \infty.
\end{equation}
The explicit form of eigenvectors of $\mathcal A$ is
$\varphi_{m,k}=\left((I-{\rm e}^{\lambda_m^k}A_{-1})x_{m,k},\: {\rm e}^{\lambda_m^k \theta}x_{m,k}\right)^T$,
where $x_{m,k}\in {\rm Ker}\Delta_{\mathcal A}(\lambda_m^k)$.

Outside the circles $L^{k}_m, \ |k| >N, \ m=1, \dots, n$, there is only a finite number of eigenvalues of ${\mathcal A}$,
which we denote by $\widehat \lambda_s$, $s=1,\ldots, \ell_N$ counted with multiplicities.
The corresponding generalized eigenvectors of the operator ${\mathcal A}$ we denote by $\widehat \varphi_s$.
The family
\begin{equation}\label{ctr_eqn_basis_A}
\{\varphi\}=\{\varphi_{m,k}\}\cup \{\widehat \varphi_s\}
\end{equation}
forms a Riesz basis of the space $M_2$ (\cite{Rabah_Sklyar_Rezounenko_2005}).

We denote by
\begin{equation}\label{ctr_eqn_basis_A_adj}
\{\psi\}=\{\psi_{m,k}\}\cup \{\widehat \psi_s\}
\end{equation}
the family of eigenvectors of the adjoint operator ${\mathcal A}^*$, which is biorthogonal to $\{\varphi\}$.
Here ${\mathcal A}^* \psi_{m,k} = \overline{\lambda_m^k} \psi_{m,k}$, $m=1,\ldots,n$, $|k|> N$ and $s=1,\ldots, \ell_N$.
The explicit form of eigenvectors of the adjoint operator ${\mathcal A}^*$ is
\begin{equation}\label{ctr_eqn_eigenvector}
\psi_{m,k}=\left(y_{m,k},\:
\left[\overline{\lambda_m^k} {\rm e}^{-\overline{\lambda_m^k} \theta}I -A_2^*(\theta)
+ \int_0^\theta {\rm e}^{\overline{\lambda_m^k} (s-\theta)} \left(A_3^*(s)+ \overline{\lambda_m^k} A_2^*(s)\right)\:{\rm d}s
\right]y_{m,k}\right)^T,
\end{equation}
where $y_{m,k}\in {\rm Ker}\Delta_{\mathcal A}^*(\overline{\lambda_m^k})$.

The family (\ref{ctr_eqn_basis_A_adj})  forms a Riesz basis of the space $M_2$.
The proofs of the propositions mentioned in this section may be found in
\cite{Rabah_Sklyar_Rezounenko_2003,Rabah_Sklyar_Rezounenko_2005,Rabah_Sklyar_Rezounenko_2008}.

\vskip3mm

Let us pose the controllability problem as a moment problem.
To do this, we expand the steering condition
$x_T=\left(\begin{array}{c} y_T \\ z_T(\cdot)\end{array}\right)=\int\limits_0^T {\rm e}^{{\mathcal A} t}{\mathcal B} u(t)\: {\rm d} t$
with respect to the biorthogonal bases $\{\varphi\}$ and $\{\psi\}$ given by (\ref{ctr_eqn_basis_A}) and (\ref{ctr_eqn_basis_A_adj}).
A state $x = (y,\: z(\cdot)) \in M_2$ is reachable at a time $T$ if and only if
$$
\sum_{\varphi \in\{\varphi\}}\left \langle x, \psi\right \rangle \varphi
= \sum_{\varphi \in\{\varphi\}}\int\nolimits_0^T\left\langle {\rm e}^{{\mathcal A}t}{\mathcal B}u(t),
\psi\right \rangle{\;{\rm d}} t \cdot \varphi.
$$
Let $\{b_1, \dots, b_r \}$ be an arbitrary basis of the image of the matrix $B$
and ${\mathbf b}_d=(b_d,\: 0)^T \in M_2$, $d=1,\ldots,r$.
Then the steering condition is equivalent to the following system of equalities:
\begin{equation}\label{ctr_eqn_2.1}
\begin{array}{rcl}
\left\langle x_T, \psi\right\rangle & = & \int\limits_0^T\left\langle {\rm e}^{{\mathcal A}t}{\mathcal B}u(t),
 \psi\right\rangle{\;{\rm d}} t\\
& =& \sum\limits_{d=1}^r\int\limits_0^T\left\langle {\rm e}^{{\mathcal A}t}{\mathbf {b}}_d,\psi\right\rangle u_d(t){\;{\rm d}} t,
\end{array}
\end{equation}
where $\psi \in\{\psi\}$, $u(\cdot) \in L_2(0,T; \mathbb{C}^r)$.
Using the representation (\ref{ctr_eqn_eigenvector}) for eigenvectors $\psi=\psi_{m,k}$, $m=1,\ldots,n$, $|k|>N$,
we obtain the following identity:
\begin{equation}\label{ctr_eqn_2.3}
\left \langle {\rm e}^{{\mathcal A}t}{\mathbf {b}}_d, \: \psi_{m,k}\right \rangle_{M_2}
={\rm e}^{{\lambda_m^k} t} \left \langle {\mathbf {b}}_d, \psi_{m,k}\right \rangle _{M_2}
= {\rm e}^{{\lambda_m^k} t} \left \langle b_d, y_{m,k}\right \rangle_{\mathbb{C}^n},
\end{equation}
where $y_{m,k}\in{\rm Ker}\Delta_{\mathcal A}^*(\overline{\lambda_m^k})$.
Let us introduce the notation:
\begin{equation} \label{ctr_eqn_2.18}
q_{m,k}^{d}= k \left \langle {\mathbf b}_d, \psi_{m,k} \right \rangle_{M_2}.
\end{equation}

Due to (\ref{ctr_eqn_2.3}), the equalities (\ref{ctr_eqn_2.1}) corresponding to $\psi \in \{\psi_{m,k},\; |k|>N, m=1,\ldots,n\}$
take the form:
\begin{equation} \label{ctr_eqn_2.20}
k \left \langle x_T, \psi_{m,k}\right\rangle
=\sum_{d=1}^r\int\nolimits_0^T {\rm e}^{\lambda_m^k t}q_{m,k}^{d} u_d(t){\;{\rm d}} t.
\end{equation}

Besides, for generalized eigenvectors $\psi=\widehat \psi_s$, $s=1,\ldots,\ell_N$, the following relations hold:
\begin{eqnarray*}
\left \langle {\rm e}^{{\mathcal A}t}\mathbf b_d, \psi\right \rangle,
= \left \langle \mathbf b_d, {\rm e}^{{\mathcal A}^*t}\psi \right \rangle
=\widehat q_s^{d}(t){\rm e}^{\widehat \lambda_s t},
\end{eqnarray*}
where $\widehat q_s^{d}(t)$ are polynomials of appropriate degrees.
Therefore, the equalities (\ref{ctr_eqn_2.1}) corresponding to $\psi \in \{\widehat \psi_s\}$ take the form:
\begin{equation}\label{ctr_eqn_2.21}
\left \langle x_T, \widehat \psi_s\right\rangle
=\sum_{d=1}^r\int\nolimits_0^T {\rm e}^{\widehat \lambda_s t}\widehat q_{s}^{d}(t)u_d(t){\;{\rm d}} t.
\end{equation}

Thus, a state $x_T\in M_2$ is reachable from $0$ at the time $T>0$ if and only if
for some controls $u_d(\cdot) \in L_2(0,T), \ d=1, \dots,r$ the equalities (\ref{ctr_eqn_2.20}) and (\ref{ctr_eqn_2.21}) hold.

The obtained moment problem (\ref{ctr_eqn_2.20})--(\ref{ctr_eqn_2.21}) is the main object of our further
analysis. We conclude the section by two estimates which are important for the further analysis.
\begin{lem}\label{ctr_lem_lem2}
There exists a constant $\delta_1>0$ such that
\begin{equation}\label{ctr_eqn_2.11}
|q_{m,k}^{d}| \le \delta_1, \quad m=1,\ldots,n, |k|>N, d=1,\ldots,r.
\end{equation}
\end{lem}

\begin{lem}\label{ctr_lem_lem1}
There exists a sequence $\{\alpha_k\}$, ${\sum\limits_{|k|>N}}\alpha_k^2 <+\infty$, such that
for all $m=1,\ldots,n$, $|k| > N$, $d=1,\ldots,r$ and $t\in [0,T]$ the following estimate holds:
\begin{equation} \label{ctr_eqn_2.4b}
\left| {\rm e}^{\lambda_m^k t} \left \langle {\mathbf b}_d, \psi_{m,k} \right \rangle_{M_2}
- {\rm e}^{\widetilde{\lambda}_m^k t} \left \langle {\mathbf b}_d, \widetilde{\psi}_{m,k} \right \rangle_{M_2} \right|
\le \frac{\alpha_k}{|k|}.
\end{equation}
\end{lem}
The proofs of these propositions may be found in \cite{Rabah_Sklyar_2007}.

\section{Necessary conditions of controllability}\label{sec:necess}

Let us investigate the solvability of the equations (\ref{ctr_eqn_2.20})--(\ref{ctr_eqn_2.21}).
The following well-known result is a consequence of Bari theorem (see \cite{Gohberg_Krein_1969},\cite{Young_1980}).
\begin{lem}\label{ctr_ut_prop3}
Consider the following moment problem:
\begin{equation}\label{ctr_eqn_3.1}
 s_k=\int\nolimits_0^Tg_k(t)u(t){\;{\rm d}} t, \quad T>0, \; k\in\mathbb{N},
\end{equation}
where $g_k(\cdot)\in L_2(0,T)$ for all $k \in \mathbb{N}$.
The following statements are equivalent:
\begin{enumerate}
\item[{\rm (i)}] For the sequence $\{s_k\}_{k\in\mathbb{N}}$ the problem (\ref{ctr_eqn_3.1})
has a solution $u(\cdot) \in L_2(0,T)$ if and only if $\{s_k\} \in \ell_2$,
i.e. ${\sum\limits_{k\in \mathbb{N}}s_k^2} <+\infty$;
\item[{\rm (ii)}] the family $\{g_k(t)\}_{k\in \mathbb{N}}, \ t \in [0,T]$ forms a Riesz basis in the closure of its linear span
$$
{\mathrm {Cl}\,} {\mathrm{Lin}\{g_k(t), \ {k\in \mathbb{N}}\}}.
$$
\end{enumerate}
\end{lem}

The following propositions on solvability of moment problem was proved in \cite{Rabah_Sklyar_2007}.
\begin{lem} \label{ctr_ut_prop4}
Let us suppose that for some $T_1>0$ the functions
$\{g_k(t)\}_{k\in \mathbb{N}}$, defined on $[0,T_1]$, form a Riesz basis in
${\mathrm {Cl}\,} {\mathrm{Lin}\{g_k(t), \ {k\in \mathbb{N}}\} }\subset L_2(0,T_1)$
and $\mathrm {codim}\: {\mathrm {Cl}\,} {\mathrm{Lin}\{g_k(t), \ {k\in \mathbb{N}}\} }<+\infty$.
Then for any $T$: $0<T<T_1$, there exists an infinite-dimensional subspace $\ell_T \subset \ell_2$, such
that the moment problem {\rm (\ref{ctr_eqn_3.1})} is unsolvable on $[0,T]$ for $\{s_k\} \in \ell_T\backslash \{0\}$.
\end{lem}

\begin{lem}\label{ctr_ut_prop5}
Let us consider the moment problem
\begin{equation}\label{ctr_eqn_3.3}
s_k=\sum_{d=1}^r\int\nolimits_0^Tg_k^d(t)u_d(t){\;{\rm d}} t, \qquad k \in \mathbb{N},
\end{equation}
assuming that
$\sum\limits_{k \in \mathbb{N}}\int_0^T\vert g_k^d(t)\vert^2{\;{\rm d}} t <+\infty$ for all $d=1,\ldots,r$.

Then the set $S_{0,T}$ of sequences $\{s_k\}$ for which the problem {\rm (\ref{ctr_eqn_3.3})}
is solvable is a nontrivial submanifold of $\ell_2$, i.e. $S_{0,T} \ne \ell_2$.
\end{lem}

The following proposition (see \cite{Rabah_Sklyar_2007}) shows that the reachability set ${\mathcal R}_{T}$
is always a subset of ${\mathcal D}({\mathcal A})$ (see also \cite{Ito_Tarn_1985}).
\begin{lem}\label{ctr_ut_prop6}
If the state $x_T=\left(\begin{array}{c} y_T \\ z_T(\cdot)\end{array}\right)$
is reachable from $0$ by the system {\rm (\ref{ctr_eqn_*})},
then it satisfies the following equivalent conditions:
\begin{enumerate}
\item[{\rm (C1)}] $\sum\limits_{\vert k\vert>N}\sum\limits_{m=1}^n
k^2\left\vert \left\langle \left(\begin{array}{c} y_T\\ z_T(\cdot)\end{array}\right),
\: \psi_{m,k} \right\rangle \right\vert^2 < \infty$;
\label{ctr_ut_prop6_C1}
\item[{\rm (C2)}] $\sum\limits_{\vert k\vert>N}\sum\limits_{m=1}^n
k^2 \left\Vert P_m^{(k)} \left(\begin{array}{c} y_T\\ z_T(\cdot)\end{array}\right) \right\Vert^2 < \infty$;
\label{ctr_ut_prop6_C2}
\item[{\rm (C3)}] $\left(\begin{array}{c} y_T\\ z_T(\cdot)\end{array}\right) \in {\mathcal D}({\mathcal A})$.
\end{enumerate}
\end{lem}

Let us prove necessity of the conditions {\rm (i)} and {\rm (ii)} of Theorem~\ref{ctr_thr_intr} for controllability.

\begin{thm}\label{ctr_thr_thm4}
If the condition {\rm (i)} of Theorem~\ref{ctr_thr_intr} is not verified, i.e.
there exist $\lambda \in \mathbb{C}$ and $y \in  \mathbb{C}^n\backslash \{0\}$, such
 that $\Delta_{\mathcal A}^*(\lambda)y=0$ and $ B^*y=0$,
then the system {\rm (\ref{ctr_eqn_1s})} is not  controllable at any time $T>0$.
\end{thm}
\begin{proof}
The condition {\rm (i)} may be reformulated as follows:
there is no eigenvector $g$ of the operator ${\mathcal A}^*$ belonging to ${\rm Ker}\:{\mathcal B}^*$.
This assertion follows from the explicit form (\ref{ctr_eqn_eigenvector}) of eigenvectors of ${\mathcal A}^*$.

Assume that there exists a vector $g\ne 0$ such that ${\mathcal {A}}^*g=\lambda g$ and $g \in {\rm Ker}{\mathcal {B}}^*$.
For an arbitrary state $x_T \in {\mathcal {R}}_{T}$ the following equality holds:
$$
\left\langle x_T, \:g\right\rangle
=\int\nolimits_0^T \left \langle u(t),{\mathcal {B}}^*{\rm e}^{{\mathcal {A}}^*t}g\right\rangle{\;{\rm d}} t =0.
$$
This means that for any $T>0$ the reachability set ${\mathcal {R}}_{T}$ is not dense in $M_2$ and, therefore,
is not equal to ${\mathcal {D}}({\mathcal {A}})$ which is dense in $M_2$
since ${\mathcal {A}}$ is an infinitesimal generator. Thus, the system is not  controllable.
\end{proof}

Further we show that controllability of the pair $(A_{-1}, B)$ is a necessary condition of
 controllability of the system~(\ref{ctr_eqn_1s}). We prove this assertion in two situations:
singular and nonsingular matrix $A_{-1}$.
\begin{thm}\label{ctr_thr_thm63}
If the condition {\rm (ii)} of Theorem~\ref{ctr_thr_intr}  is not verified, i.e. the pair $(A_{-1}, B)$ is not controllable,
then the system ~(\ref{ctr_eqn_1s}) is not controllable as well.
\end{thm}
\begin{proof}
If the pair $(A_{-1}, B)$ is not controllable then there exist $\mu_0\in \sigma(A_{-1})$ and $v_0\in \mathbb{C}^n\backslash \{0\}$
such that $A_{-1}^*v_0=\overline{\mu_0}v_0$  and $B^*v_0=0$.

We begin with the case when $\mu_0=0$ is an uncontrollable eigenvalue of $A_{-1}$, i.e.
\begin{equation}\label{ctr_eq63}
A_{-1}^*v_0=0 \; \mbox{  and  } \; B^*v_0=0.
\end{equation}
Let us premultiply the equation~(\ref{ctr_eqn_1s}) by the vector $v_0^*$:
$$
v_0^* \dot{z}(t)=v_0^* A_{-1}\dot{z}(t-1)
+\int_{-1}^0 \left[v_0^* A_2(\theta) \dot{z}(t+\theta) + v_0^* A_3(\theta) z(t+\theta)\right] \: {\rm d}\theta
+ v_0^* Bu.
$$
Taking into account the relations (\ref{ctr_eq63}), we obtain the following equality:
\begin{equation}\label{ctr_eq64}
v_0^* \dot{z}(t)=
\int_{-1}^0 \left[v_0^* A_2(\theta) \dot{z}(t+\theta) + v_0^* A_3(\theta) z(t+\theta)\right] \: {\rm d}\theta.
\end{equation}
If we suppose that the system~(\ref{ctr_eqn_1s}) is  controllable at a time $T>0$
then the set of its solutions under different controls should coincide with the space $H^1(T-1,T; \mathbb{C}^n)$.
The last means that
$$\{ v_0^* \dot{z}(t), \; t\in [T-1,T]\}=L_2(T-1,T; \mathbb{C}).$$

On the other hand, the operator $Q(z)=\int_{-1}^0 \left[v_0^* A_2(\theta) \dot{z}(t+\theta) +
v_0^* A_3(\theta) z(t+\theta)\right] \: {\rm d}\theta$, which acts from $H^1(T-2,T; \mathbb{C}^n)$ to $L_2(T-1,T; \mathbb{C})$,
is Fredholm operator. Indeed, changing time variable $\tau=t+\theta$, we obtain
$$
Q(z)=\int_{t-1}^t \left[v_0^* A_2(\tau-t) \dot{z}(\tau) + v_0^* A_3(\tau-t) z(\tau)\right] \: {\rm d}\tau.
$$
Hence, the operator $Q$ is compact and, thus, its image
  does not coincide with
the whole space $L_2(T-1,T; \mathbb{C})$.
The  obtained contradiction  proves the theorem in the case $\mu_0=0$.

Now let us consider the case when only nonzero eigenvalues of $A_{-1}$ are uncontrollable.
Without loss of generality we may assume that $\det A_{-1}\not=0$.
Indeed, since $0\in\sigma(A_{-1})$ is controllable eigenvalue,  there is a matrix $P$ such that the matrix
$A_{-1}+BP$ is nonsingular (see \cite{Wonham_1985}). Obviously, the pair $(A_{-1}+BP, B)$ remains uncontrollable.
Then using a change of control
we obtain an equivalent controllability problem for system with neutral nonsingular matrix $A_{-1}+BP$.

Since  $A_{-1}$ is nonsingular, then the moment equalities (\ref{ctr_eqn_2.20})--(\ref{ctr_eqn_2.21}) hold.
Consider an uncontrollable eigenvalue $\mu_{m_0}$ of $A_{-1}$ ($A_{-1}^*v_0=\bar \mu_{m_0} v_0$, $B^*v_0=0$)
and the subset of (\ref{ctr_eqn_2.20}) which corresponds to $m=m_0$:
\begin{eqnarray}\label{ctr_eqn_4.6}
s_k=k \left \langle x_T, \psi_{m_0,k}\right\rangle
=\sum_{d=1}^r\int_0^T {\rm e}^{\lambda_{m_0}^k t}q_{m_0,k}^{d} u_d(t){\;{\rm d}} t, \quad \vert k\vert >N,
\end{eqnarray}
where  $q_{m_0,k}^{d}= k \left \langle {\mathbf b}_d, \psi_{m_0,k} \right \rangle_{M_2}$.
Let us show that there exist sequences $\{s_k\} \in \ell_2$ for which the moment problem (\ref{ctr_eqn_4.6}) is unsolvable.

For $m=m_0$, corresponding eigenvectors of $\widetilde{\mathcal A}$ are of the form
$\widetilde{\psi}_{m_0,k}=\left(v_0,\: \overline{\lambda_{m_0}^k} {\rm e}^{-\overline{\lambda_{m_0}^k} \theta} v_0\right)^T$,
what implies that
$\left \langle {\mathbf b}_d, \widetilde{\psi}_{m_0,k} \right \rangle_{M_2}
=\left \langle  b_d, v_0 \right \rangle_{\mathbb{C}^n}=0$
for all $d=1,\ldots,r$ and $|k|>N$.
Applying Lemma~\ref{ctr_lem_lem1}, we obtain the following estimate:
\begin{equation}\label{ctr_eqn_4.77}
\sum_{|k|>N} k^2 \sum_{d=1}^r\int_0^T \left| {\rm e}^{\lambda_m^k t}
\left \langle {\mathbf b}_d, \psi_{m,k} \right \rangle_{M_2} \right|^2\: {\rm d} t < +\infty.
\end{equation}
It follows from Lemma~\ref{ctr_ut_prop5} that the solvability set for the system~(\ref{ctr_eqn_4.6})
is a nontrivial linear manifold $\ell_T \subset \ell_2$, $\ell_T \ne \ell_2$ for any time $T>0$.
In other words, there exist sequences $\{s_k\}_{\vert k\vert > N}$ for which the system of equalities (\ref{ctr_eqn_4.6}) is
not solvable.
The last means, that there exist states $x_T$ which satisfy the condition (C1),
but which are not reachable from $0$ by virtue of the system~(\ref{ctr_eqn_1s}).
Thus, ${\mathcal R}_{T} \ne {\mathcal D}({\mathcal A})$ for any $T>0$.
The contradiction obtained completes the proof of the theorem.
\end{proof}

\section{Sufficiency in the case of one-dimensional control}\label{sec:suff1}

In the case of systems with one-dimensional control ($r=1$, $B=b\in \mathbb{C}^{n\times 1}$)
the moment problem (\ref{ctr_eqn_2.20})--(\ref{ctr_eqn_2.21}) takes the following form:
\begin{equation} \label{ctr_eqn_2.20a}
\alpha_{m,k} \left \langle x_T, \psi_{m,k}\right\rangle
=\int_0^T {\rm e}^{\lambda_m^k t}u(t) \;{\rm d} t,\quad |k|>N, m=1,\ldots,n,
\end{equation}
\begin{equation}\label{ctr_eqn_2.21a}
\left \langle x_T, \widehat \psi_s\right\rangle
=\int\nolimits_0^T {\rm e}^{\widehat \lambda_s t}\widehat q_{s}(t)u(t){\;{\rm d}} t,\quad s=1,\ldots, \ell_N,
\end{equation}
where $N$ is big enough, the family (\ref{ctr_eqn_2.20a}) is infinite,
$\widehat q_{j}$ are polynomials, the family (\ref{ctr_eqn_2.21a}) is finite, and
$\alpha_{m,k}=\left(\left \langle {\mathbf b}, \psi_{m,k} \right \rangle_{M_2}\right)^{-1}$, ${\mathbf b}=(b,0)^T$.

From Lemma~\ref{ctr_lem_lem2} and the explicit form of the basis $\{\psi\}$ of the operator ${\mathcal A}^*$ it follows
that for all $m=1,\ldots,n$ and $k: |k|>N$ the following estimate holds:
$$0<C_1\le\left|\frac{1}{k}\alpha_{m,k}\right|\le C_2<+\infty.$$

Our next objective is to find the conditions for the families $\{{\rm e}^{\lambda_m^k t}\}$
and  $\{{\rm e}^{\widehat \lambda_s t}\widehat q_{s}(t)\}$ to form a Riesz basis of the closure of its linear span.

Let $\delta_1, \dots, \delta_n$ be different, modulus $2\pi {\rm i}$, complex numbers,
and let $N\in\mathbb{N}$ be natural integer,
and let the set $\{\varepsilon_{m,k},\; \vert k\vert>N, m=1,\ldots,n \}\subset\mathbb{C}^n$
be such that $\sum\limits_{m,k}|\varepsilon_{m,k}|^2<+\infty$.
We denote by ${\mathcal E}_N$ the following (infinite) family of functions:
$$
{\mathcal E}_N=\left \{ {\rm e}^{(\delta_m+2\pi {\rm i}  k
+ \varepsilon_{m,k}) t},\quad {\vert k\vert>N}, m=1, \dots,n\right\}.
$$
Next, let $\varepsilon_1, \dots, \varepsilon_r$ be another collection of different complex numbers such that
 $\varepsilon_j \ne \delta_m+2\pi {\rm i}  k + \varepsilon_{m,k}$, $j=1, \dots, r$, $m=1, \dots, n$, $\vert k\vert >N$,
and let $m_1', \dots, m_r'$ be some positive integers. Let us denote by ${\mathcal E}_0$ the following (finite) family of functions
$$
{\mathcal E}_0=\left \{ {\rm e}^{\varepsilon_j t},
t {\rm e}^{\varepsilon_j t}, \dots,
t^{m_j'-1}{\rm e}^{\varepsilon_j t}\right \}_{j=1, \dots,r}.
$$
and by ${\mathcal E}$ the set  of functions ${\mathcal E}= {\mathcal E}_N\cup {\mathcal E}_0$.
\begin{thm}\label{ctr_thr_thm2}
{\rm (i)} If ${\sum\limits_{j=1}^rm'_j}=(2N+1)n$, then the family
${\mathcal E}$
forms a Riesz basis in $L_2(0,n)$.

{\rm (ii)} If $T>n$, then independently of the number of elements in ${\mathcal E}_0$
the family ${\mathcal E}$ forms a Riesz basis of the closure of its linear
span in the space $L_2(0,T)$.
\end{thm}
The proof of this theorem, based on results of \cite{Avdonin_Ivanov_1995}, may be found in \cite{Rabah_Sklyar_2007}.

Let us prove sufficiency of the controllability conditions {\rm (i)} and {\rm (ii)} of Theorem~\ref{ctr_thr_intr}.

\begin{thm} \label{ctr_thr_thm5}
Let $u\in \mathbb{C}$ ($r=1$) and the conditions {\rm (i)} and {\rm (ii)} of Theorem~\ref{ctr_thr_intr} hold.
Then
\begin{enumerate}
\item[{\rm (1)}] the system~(\ref{ctr_eqn_1s}) is controllable at any time $T>n$;
\item[{\rm (2)}] the estimation of the critical time of controllability is exact,
i.e. the system~(\ref{ctr_eqn_1s}) is uncontrollable at any time $T\le n$.
\end{enumerate}
\end{thm}
\begin{proof}
Let us note that dimensions of all eigenspaces (corresponding to different eigenvalues) of ${\mathcal {A}}^*$ are equal to $1$.
Indeed, otherwise there exists an eigenvector $g$ of the operator ${\mathcal {A}}^*$,
such that  $\langle {\mathbf b},g\rangle_{M_2}=0$.
Since $g=(y,\: z(\theta))^T$, where $y$: $\Delta^*_{{\mathcal {A}}}(\lambda_0)y=0$, $\lambda_0\in\sigma({\mathcal {A}}^*)$,
and since $\langle {\mathbf b},g\rangle_{M_2}=\langle b,y\rangle_{\mathbb{C}^n}$,
we obtain a contradiction with the condition {\rm (i)}.

Let us consider the problem (\ref{ctr_eqn_2.20a})--(\ref{ctr_eqn_2.21a}).
From condition {\rm (i)} it follows that $\left \langle {\mathbf b}, \psi_{m,k} \right \rangle_{M_2}\not=0$ for all $m$ and $k$.
Moreover, all polynomial $\{\widehat q_s(t)\}$, $s=1,\ldots,\ell_N$ are nontrivial.
By the moment problem we construct the following families of functions:
$$
\begin{array}{lcl}
\Phi_1 & = & \left\{ {\rm e}^{\lambda_m^{k}t}, \ \vert k\vert>N, \ m=1,\dots,n \right\}, \\
\Phi_2 & = & \left\{ {\rm e}^{\widehat \lambda_s t}\widehat q_{s}(t), \ s=1,\dots,\ell_N\right\}.
\end{array}
$$
Due to Theorem~\ref{ctr_thr_thm2} for $T>n$ and for big enough $N$, the family
$$
\Phi=\Phi_1 \bigcup \Phi_2
$$
constitutes a Riesz basis in ${\mathrm {Cl}\,}{{\mathrm {Lin}\,}\Phi} \subset L_2(0,T)$.
Thus, due to Lemma~\ref{ctr_ut_prop3}, the moment problem~(\ref{ctr_eqn_2.20a})--(\ref{ctr_eqn_2.21a})
is solvable if and only if the right-hand side is an element of $\ell_2$,
or, equivalently the condition (C1) from Lemma~\ref{ctr_ut_prop6} holds.
Since (C1) is equivalent to (C3) we conclude that for $T>n$ the moment problem is solvable
if and only if $x_T\in {\mathcal D}({\mathcal A})$, i.e. ${\mathcal {R}}_{T}={\mathcal {D}}({\mathcal {A}})$.

To prove the assertion~(2) we remind that the number of elements in family $\Phi_2$ equals to $\ell_N=(2N+2)n$.
On the other hand, it follows from Theorem~\ref{ctr_thr_thm2} that in $L_2(0,n)$ one has
$$
{\mathrm {codim}\,} {\mathrm {Cl}\,} {\mathrm {Lin}\,} \Phi_1 = (2N+1)n.
$$
Thus, the family $\Phi=\Phi_1\cup \Phi_2$ contains $n$ functions,
which can be represented as linear combination of other functions from this family.
This means that the codimension ${\mathcal {R}}_T$ in ${\mathcal {D}}({\mathcal {A}})$
is not equal to zero: ${\mathrm {codim}\,} {\mathcal {R}}_T=n$.
Hence, the reachability set ${\mathcal {R}}_T$ for $T=n$ is not equal to ${\mathcal {D}}({\mathcal {A}})$
and the system is not controllable.
For $T<n$ it follows from Lemma~\ref{ctr_ut_prop4} that the codimension of the set
${\mathcal {R}}_T$ in ${\mathcal {D}}({\mathcal {A}})$ is infinite.
\end{proof}

\section{Sufficient conditions: the multivariable case}\label{sec:suff2}

Consider the case $\dim B=r>1$. Without loss of generality we assume that the pair $({A}_{-1}, {B})$
is in Frobenius normal form, i.e.
$A_{-1}={\rm diag}\{F_1, \ldots, F_r\}$, $\dim F_i=s_i$, and $F_i$ are of the form~(\ref{ctr_eqn_frobenius});
${B}={\rm diag}\{g_1,\dots,g_r\}$, where $g_i=(0,\:0,\ldots, 1)^{\mathrm T}\in\mathbb{C}^{s_i}$.
It is well-known that
\begin{equation} \label{ctr_eqn_2.20bb}
\max\limits_i \dim F_i =n_1,
\end{equation}
where $n_1$ is the first controllability index of the pair $({A}_{-1}, {B})$, i.e. $n_1$ is the minimal integer $\nu$
satisfying the relation
${\mathrm {rank}\,} (B,\; A_{-1}B,\: \ldots, \: A_{-1}^{\nu-1}B)=n$.

According to the representation in Frobenius form, we can rewrite the infinite part~(\ref{ctr_eqn_2.20}) of the moment problem
as follows:
\begin{equation} \label{ctr_eqn_2.20b}
\begin{array}{c}
k \left \langle x_T, \psi_{m,k}\right\rangle
=\int\limits_0^T {\rm e}^{\lambda_m^k t}q_{m,k}^{1} u_1(t){\;{\rm d}} t+
\sum\limits_{d\not=1}\int\limits_0^T {\rm e}^{\lambda_m^k t}q_{m,k}^{d} u_d(t){\;{\rm d}} t, \quad m\in S_1,\\
k \left \langle x_T, \psi_{m,k}\right\rangle
=\int\limits_0^T {\rm e}^{\lambda_m^k t}q_{m,k}^{2} u_2(t){\;{\rm d}} t+
\sum\limits_{d\not=2}\int\limits_0^T {\rm e}^{\lambda_m^k t}q_{m,k}^{d} u_d(t){\;{\rm d}} t, \quad m\in S_2,\\
\ldots\\
k \left \langle x_T, \psi_{m,k}\right\rangle
=\int\limits_0^T {\rm e}^{\lambda_m^k t}q_{m,k}^{r} u_r(t){\;{\rm d}} t+
\sum\limits_{d\not=r}\int\limits_0^T {\rm e}^{\lambda_m^k t}q_{m,k}^{d} u_d(t){\;{\rm d}} t, \quad m\in S_r,
\end{array}
\end{equation}
where $S_1=\{1,\ldots,s_1\}$, $S_2=\{s_1+1,\ldots, s_1+s_2\}$, ..., $S_r=\{s_1+\ldots,s_{r-1}+1,\ldots, n\}$.

Next we apply Theorem~\ref{ctr_thr_thm2} to the family of functions from~(\ref{ctr_eqn_2.20b}).
Let us fix $d\in\{1,\dots,r\}$ and choose an arbitrary subset of $L\subset \{1,\dots,n\}$.

\begin{thm}\label{ctr_thr_thm3}
For arbitrary $d$, $L$, and for all $T > n'=|L|$ the set
\begin{eqnarray*}
\Phi_1=\left \{ {\rm e}^{\lambda_m^{k}t}q_{m,k}^{d}, \quad \vert k\vert>N; \ m \in L \right\}
\end{eqnarray*}
constitutes a Riesz basis of the closure of its linear span ${\mathrm {Cl}\,} {\mathrm {Lin}\Phi_1}$ in $L_2(0,T)$.

If $T=n'$, then ${\rm codim}\: {\mathrm {Cl}\,} {\mathrm {Lin}\Phi_1}=(2N+1)n'$ in the space $L_2(0,n')$.
\end{thm}
\begin{proof}
Let us consider the linear operator ${\mathcal {T}}: {\mathrm {Lin}\,} \Phi_1 \to {\mathrm {Lin}\,}\Phi_1$ defined
on elements of $\Phi_1$ by the following relations
$$
{\mathcal {T}}({\rm e}^{\lambda_m^{k}t}q_{m,k}^{d})={\rm e}^{\lambda_m^{k}t}, \quad \vert k\vert>N, m \in L.
$$
Due to Lemma~\ref{ctr_lem_lem2} the family $\{q_{m,k}^{d}\}$ is uniformly bounded.
Thus, from Theorem~\ref{ctr_thr_thm2} we obtain that the operator ${\mathcal {T}}$ is bounded in the sense of $L_2(0,T)$
and its extension to $L={\mathrm {Cl}\,}{{\mathrm {Lin}\,}\Phi_1}$ is a bounded one-to-one operator from $L$ to $L$.

Hence, since due to Theorem~\ref{ctr_thr_thm2} the images of the elements of $\Phi_1$  form a Riesz basis of $L$,
then $\Phi_1$ is also a Riesz basis of $L$ in $L_2(0,T)$.
\end{proof}

We also need the following result (see \cite[Theorem 5.5]{Rabah_Sklyar_2007}).
\begin{thm}\label{ctr_thr_cor}
Consider the system~{\rm (\ref{ctr_eqn_*})} and suppose that there are an integer $N$
and a time $T_0>0$ such that the moment problem~{\rm (\ref{ctr_eqn_2.20})}
is solvable for $T=T_0$ and all sequences
$\left \{ k \left \langle x_T, \psi_{m,k}\right\rangle\right \}_{\vert k\vert >N}$ satisfying {\rm (C1)}.

Then, from the condition {\rm (i)} of Theorem~\ref{ctr_thr_intr}
it follows that ${\mathcal {R}}_T={\mathcal {D}}({\mathcal {A}})$ for $T>T_0$.
\end{thm}

Now we prove the main result of this section.
\begin{thm}\label{ctr_thr_thm6}
Let the conditions {\rm (i)} and {\rm (ii)} of Theorem~\ref{ctr_thr_intr} hold for a system of the form~(\ref{ctr_eqn_1s}).
Then the system~(\ref{ctr_eqn_1s}) is controllable, and, moreover, the critical time of
controllability is $T_0=n_1$, where $n_1$ is the first controllability index of the pair $(A_{-1}, B)$.
\end{thm}

\begin{proof}
We assume that the pair $(A_{-1}, B)$ is in Frobenius normal form.
Then for all $i=1,\ldots, r$, $m\in S_i$, $d\not=i$ and for all $|k|>N$ the following relation holds:
\begin{equation}\label{ctr_eqn_mth1}
\left \langle {\mathbf b}_d, \widetilde{\psi}_{m,k} \right \rangle_{M_2}=\left \langle  b_d, c_{m} \right \rangle_{\mathbb{C}^n}=0,
\end{equation}
where $c_m$: $A_{-1}c_m=\mu_m c_m$. Thus for all $i=1,\ldots,r$ and $m\in S_i$ the following equality holds
\begin{equation}\label{ctr_eqn_mth2}
\sum\limits_{d\not=i}\int\limits_0^T {\rm e}^{\lambda_m^k t}q_{m,k}^{d} u_d(t){\;{\rm d}} t
= \sum\limits_{d\not=i}\int\limits_0^T k\left({\rm e}^{\lambda_m^k t} \left \langle {\mathbf b}_d, \psi_{m,k} \right \rangle_{M_2}
- {\rm e}^{\widetilde{\lambda}_m^k t} \left \langle {\mathbf b}_d,
\widetilde{\psi}_{m,k} \right \rangle_{M_2} \right) u_d(t){\;{\rm d}} t.
\end{equation}

For any $N\in\mathbb{N}$ the moment problem~(\ref{ctr_eqn_2.20b}) may be written in operator form
$$
\{S_{m,k}\}=Z_N u(\cdot)+Q_N u(\cdot),
$$
where
$\{S_{m,k}\}=\{k \left \langle x_T, \psi_{m,k}\right\rangle\}$
and the operators $Z_N, Q_N: L_2(0,T;\mathbb{C}^r)\rightarrow \ell_2$ are of the form
$$
\begin{array}{l}
Z_N u(\cdot)=\left\{\int\limits_0^T {\rm e}^{\lambda_m^k t}q_{m,k}^{i} u_i(t){\;{\rm d}} t,\; |k|>N \right\},\\
Q_N u(\cdot)=\left\{\sum\limits_{d\not=i}\int\limits_0^T k\left({\rm e}^{\lambda_m^k t}
\left \langle {\mathbf b}_d, \psi_{m,k} \right \rangle_{M_2}
- {\rm e}^{\widetilde{\lambda}_m^k t}  \langle {\mathbf b}_d, \widetilde{\psi}_{m,k}
\rangle_{M_2} \right) u_d(t){\;{\rm d}} t,\; |k|>N \right\}.
\end{array}
$$

Due to Theorem~\ref{ctr_thr_thm3}, for big enough $N$ and for $T\ge n_1$, the operator $Z_N$ is surjective, i.e.
its image of the space $L_2(0,T;\mathbb{C}^r)$, $T\ge n_1$,  is the whole space $\ell_2$.
From Lemma~\ref{ctr_lem_lem1}, it follows that for big enough $N$ the operator $Q_N$ is compact,
and, moreover, $\|Q_N\|\rightarrow 0$ when $N\rightarrow +\infty$.

Let us show that there exists $N_0\in\mathbb{N}$ such that for all $N>N_0$ one has:
$${\rm Im} [Z_N+Q_N]=\ell_2.$$
Since ${\rm Im} Z_N=\ell_2$, then there exists a constant $\gamma_N>0$
such that $\Vert Z_N^* x\Vert \ge \gamma_N\Vert x\Vert$ for all $x\in \ell_2$ (see, e.g. \cite[Theorem 4.13]{Rudin_1975}).
For $N>N_0$ we introduce the notation
$\ell_2^N=\{\{S_{m,k}\}_{\vert k\vert > N}:\; \sum_{\vert k\vert> N} \vert s_k\vert^2 <+ \infty\}$.
Thus, we have $Z_N=PZ_{N_0}$, where projectors $P: \ell_2^{N_0}\to \ell_2^{N}$ are defined as follows:
$$
 P(\{S_{m,k}\}_{\vert k\vert >N_0})=\{S_{m,k}\}_{\vert k\vert >N}.
$$
Therefore, $Z_N^*=Z_{N_0}^*P^*$ and $\Vert P^*x\Vert=\Vert x\Vert$, what gives
$$
 \Vert Z_N^* x\Vert =\Vert Z_{N_0}^*P^*x\Vert \ge \gamma_{N_0} \Vert x\Vert.
$$
The last means that for all $N>N_0$ and $x\in \ell_2$ the inequality $\Vert Z_N^* x\Vert \ge \gamma \Vert x\Vert$ holds,
where $\gamma=\gamma_0$.
Since $\Vert Q_N\Vert\to 0$ when $N\to +\infty$,
then choosing an appropriate $N$ we obtain estimate $\Vert Z_N-(Z_N+ Q_N)\Vert=\Vert Q_N\Vert \le \frac{\gamma}{2}$.
Thus
$$
  \Vert [Z_N^*+Q_N^*]x \Vert \ge \Vert Z_N^*x \Vert -\Vert Q_N^*x\Vert
  \ge \gamma \Vert x\Vert -\frac{\gamma}{2}\Vert x\Vert
  =\frac{\gamma}{2}\Vert x\Vert.
$$
Therefore, operator $Z_N+Q_N$ is surjective and its image equals to $\ell_2$.

Thus, the moment problem~(\ref{ctr_eqn_2.20b}) is solvable for $T\ge n_1$ and big enough $N\in\mathbb{N}$.
Applying Theorem~\ref{ctr_thr_cor}, we obtain that $R_T={\mathcal D}({\mathcal A})$ for $T>n_1$.

Arguing as in the proof of Theorem~\ref{ctr_thr_thm5}, we may show that the
codimension ${\mathcal {R}}_{n_1}$ in ${\mathcal D}({\mathcal A})$ is finite and no less than $n_1$,
what means that the system~(\ref{ctr_eqn_1s}) is uncontrollable at time $T=n_1$.
For $T<n_1$ the codimension of ${\mathcal {R}}_{T}$ in ${\mathcal D}({\mathcal A})$ is infinite.
\end{proof}

\section{Example}\label{sec:example}

Consider a three-dimensional ($n=3$) system given by the equation~(\ref{ctr_eqn_1s}) with the following coefficients:
$$
A_{-1}=\left(
\begin{array}{rrr}
-4 & 6 & -4 \\
0 & 2 & -2 \\
-3 & 3 & 2 \\
\end{array}
\right),
\qquad
B=\left(
\begin{array}{rr}
1 & 1 \\
1 & 0 \\
-1 & 1 \\
\end{array}
\right),
$$
and the matrices $A_2(\theta)$, $A_3(\theta)$ are such that  ${\rm rank} (\Delta_{\mathcal A} (\lambda)\;  B)=n$
for all $\lambda \in \mathbb{C}$.

We apply the change of control and state variables $u(t)=P\dot{z}(t-1)+v(t)$, $w=Cz$, where
$$
P=\left(
\begin{array}{rrr}
1 & -1 & 2 \\
3 & -2 & 3
\end{array}
\right),
\quad
C=\left(
\begin{array}{rrr}
1 & -1 & 1 \\
1 & 1 & 0 \\
-1 & 0 & 1
\end{array}
\right),
$$
and obtain the following system:
\begin{equation}\label{ctr_new_a-1}
\dot{w}(t) = \widehat{A}_{-1}\dot{w}(t-1)
+\int\nolimits_{-1}^0 \widehat{A}_2(\theta)\dot{w}(t+\theta)\;{\rm d}\theta
+\int\nolimits_{-1}^0 \widehat{A}_3(\theta){w}(t+\theta) \;{\rm d}\theta+\widehat{B}v,
\end{equation}
where $\widehat{A}_{-1}$ and $\widehat{B}$ are of the form
\begin{equation}\label{ctr_new_a-2}
\widehat{A}_{-1}=\left(
\begin{array}{rrr}
2 & 0 & 0 \\
0 & 0 & 1 \\
0 & 3 & 2 \\
\end{array}
\right),
\quad
\widehat{B}=\left(
\begin{array}{rr}
1 & 0 \\
0 & 0 \\
0 & 1 \\
\end{array}
\right)=(b_1, b_2).
\end{equation}

Let the operator $\mathcal A$ with eigenvalues $\lambda_m^k$
corresponds to the perturbed system (\ref{ctr_new_a-1})--(\ref{ctr_new_a-2}),
and the operator $\widetilde{\mathcal A}$ with eigenvalues $\widetilde{\lambda}_m^k$
corresponds to the system $\dot{w}(t) = \widehat{A}_{-1}\dot{w}(t-1)$.
Since the pair  $(\widehat{A}_{-1}, \widehat{B})$ is in Frobenius normal form, then
the eigenvectors $\widetilde{\psi}_{m,k}$ of the operator $\widetilde{\mathcal A}^*$ satisfy the relations
$$
\begin{array}{l}
\langle {\mathbf b}_1, \widetilde{\psi}_{m,k} \rangle = 0, \quad m=2,3\\
\langle {\mathbf b}_2, \widetilde{\psi}_{m,k} \rangle = 0, \quad m=1, \quad {\mathbf b}_i=(b_i,0)\in M_2.
\end{array}
$$
Since $q^d_{m,k}=k\langle {\mathbf b}_d, \psi_{m,k}\rangle$, where  $\psi_{m,k}$
are eigenvectors of the operator $\mathcal A^*$, then
infinite part of the moment problem (\ref{ctr_eqn_2.20}) reads as
$$
\begin{array}{l}
k \left \langle x_T,\; \psi_{1,k}\right\rangle
=\int\limits_0^T {\rm e}^{\lambda_1^k t} q^1_{1,k} u_1(t){\;{\rm d}} t
+ \int\limits_0^T f^2_{1,k} u_2(t){\;{\rm d}} t,\\
k \left \langle x_T,\; \psi_{2,k}\right\rangle
=\int\limits_0^T {\rm e}^{\lambda_2^k t} q^2_{2,k} u_2(t){\;{\rm d}} t
+ \int\limits_0^T f^2_{2,k} u_1(t){\;{\rm d}} t,\\
k \left \langle x_T,\; \psi_{3,k}\right\rangle
=\int\limits_0^T {\rm e}^{\lambda_3^k t} q^2_{3,k} u_2(t){\;{\rm d}} t
+ \int\limits_0^T f^2_{3,k} u_1(t){\;{\rm d}} t, \quad  |k|>N,
\end{array}
$$
where the functions $f^d_{m,k}$ are of the form
$$
f^d_{m,k}=k\left({\rm e}^{\lambda_m^k t} \langle {\mathbf b}_d, \psi_{m,k}\rangle -
{\rm e}^{\widetilde{\lambda}_m^k t} \langle {\mathbf b}_d, \widetilde{\psi}_{m,k}\rangle \right)
$$
and due to Lemma~\ref{ctr_lem_lem1} satisfy the estimate
$|f^d_{m,k}|\le \alpha_k$, $\sum_{k}\alpha_k^2<+\infty$.

The first controllability index $n_1$ of the pair $(\widehat{A}_{-1},\; \widehat{B})$ (or $(A_{-1},\; B)$)
equals to $2$. The conditions {\rm (i)} and {\rm (ii)} of Theorem~\ref{ctr_thr_intr} are satisfied, and, thus,
the system is controllable with the critical controllability time $T_0=2$.

\subsection*{Conclusion}
A new approach to the problem of the exact controllability by the moment problem method is proposed.
The difficulty of the choice of basis is contounred by a change of control and phase coordinates
what allows to give a more direct proof of the criterium of exact controllability.
The proposed approach offers a new challenge for controllability and stabilizability problems
for more general class of systems with neutral operator of the form $Kf=\sum_{i=1}^r A_{h_i}f(h_i)$, $h_i\in [-1,0]$.

\subsection*{Acknowledgements}
This research was partially supported by PROMEP (Mexico) via "Poyecto de Redes" and
Polish Nat. Sci. Center, grant No~514~238~438.

\def\cprime{$'$}


\begin{thebibliography}{10}

\bibitem{Avdonin_Ivanov_1995}
S.~A. Avdonin and S.~A. Ivanov.
\newblock {\em Families of exponentials}.
\newblock Cambridge University Press, Cambridge, 1995.

\bibitem{Banks_Jacobs_Langenhop_1975}
H.~T. Banks, Marc~Q. Jacobs, and C.~E. Langenhop.
\newblock Characterization of the controlled states in {$W\sb{2}\sp{(1)}$} of
  linear hereditary systems.
\newblock {\em SIAM J. Control}, 13:611--649, 1975.

\bibitem{Bartosiewicz_1983}
Z. Bartosiewicz.
\newblock A criterion of closedness of an attainable set of a delay system.
\newblock {\em Systems Control Lett.}, 3(4):211--215, 1983.

\bibitem{Bensoussan_et_al_1992}
A. Bensoussan, G. Da~Prato, M.~C. Delfour, and S.~K. Mitter.
\newblock {\em Representation and control of infinite-dimensional systems.
  {V}ol. 1}.
\newblock Birkh\"auser, Boston, MA, 1992.

\bibitem{Burns_Herdman_Stech_1983}
J.~A. Burns, T.~L. Herdman, and H.~W. Stech.
\newblock Linear functional-differential equations as semigroups on product
  spaces.
\newblock {\em SIAM J. Math. Anal.}, 14(1):98--116, 1983.

\bibitem{Gabasov-kirillova_1971}
R.~Gabasov and {{K}irillova, F.}
\newblock {\em Kachestvennaya teoriya optimalnykh protsessov}.
\newblock Izdat. ``Nauka'', Moscow, 1971.
\newblock Monographs in Theoretical Foundations of Technical Cybernetics.

\bibitem{Gohberg_Krein_1969}
I.~C. Gohberg and M.~G. Krein.
\newblock {\em Introduction to the theory of linear nonselfadjoint operators}.
\newblock  Translations of
  Mathematical Monographs, Vol. 18. American Mathematical Society, Providence,
  R.I., 1969.

\bibitem{Ito_Tarn_1985}
K. Ito and T.~J. Tarn.
\newblock A linear quadratic optimal control for neutral systems.
\newblock {\em Nonlinear Anal.}, 9(7):699--727, 1985.

\bibitem{Jacobs_Langenhop_1976}
M.~Q. Jacobs and C.~E. Langenhop.
\newblock Criteria for function space controllability of linear neutral
  systems.
\newblock {\em SIAM J. Control Optimization}, 14(6):1009--1048, 1976.

\bibitem{Khartovskii_Pavlovskaya_2013}
V.~E. Khartovski{\u\i} and A.~T. Pavlovskaya.
\newblock Complete controllability and controllability for linear autonomous
  systems of neutral type.
\newblock {\em Automation and Remote Control}, (5):769--784, 2013.

\bibitem{Manitius_Triggiani_1978}
A.~Manitius and R.~Triggiani.
\newblock Function space controllability of linear retarded systems: a
  derivation from abstract operator conditions.
\newblock {\em SIAM J. Control Optim.}, 16(4):599--645, 1978.

\bibitem{Marchenko_1979}
V.~M. Mar{\v{c}}enko.
\newblock On the controllability of zero function of time lag systems.
\newblock {\em Problems Control Inform. Theory/Problemy Upravlen. Teor.
  Inform.}, 8(5-6):421--432, 1979.

\bibitem{OConnor_Tarn_1983b}
D.~A. O'Connor and T.~J. Tarn.
\newblock On the function space controllability of linear neutral systems.
\newblock {\em SIAM J. Control Optim.}, 21(2):306--329, 1983.

\bibitem{Rabah_Sklyar_2007}
R.~Rabah and G.~M. Sklyar.
\newblock The analysis of exact controllability of neutral-type systems by the
  moment problem approach.
\newblock {\em SIAM J. Control Optim.}, 46(6):2148--2181, 2007.

\bibitem{Rabah_Sklyar_Rezounenko_2003}
R.~Rabah, G.~M. Sklyar, and A.~V. Rezounenko.
\newblock Generalized riesz basis property in the analysis of neutral type
  systems.
\newblock {\em C. R. Math. Acad. Sci. Paris}, 337(1):19--24, 2003.

\bibitem{Rabah_Sklyar_Rezounenko_2005}
R.~Rabah, G.~M. Sklyar, and A.~V. Rezounenko.
\newblock Stability analysis of neutral type systems in hilbert space.
\newblock {\em J. Differential Equations}, 214(2):391--428, 2005.

\bibitem{Rabah_Sklyar_Rezounenko_2008}
R.~Rabah, G.~M. Sklyar, and A.~V. Rezounenko.
\newblock On strong regular stabilizability for linear neutral type systems.
\newblock {\em J. Differential Equations}, 245(3):569--593, 2008.

\bibitem{Rabah_Sklyar_Springer}
R. Rabah and G. Sklyar.
\newblock On exact controllability of linear time delay systems of neutral
  type.
\newblock In {\em Applications of time delay systems}, volume 352 of {\em
  Lecture Notes in Control and Inform. Sci.}, pages 165--171. Springer, Berlin,
  2007.

\bibitem{Rabah_Sklyar_Barkhayev_2012}
R.~Rabah, G.~M. Sklyar, and P.~Yu~Barkhayev.
\newblock Stability and stabilizability of mixed retarded-neutral type systems.
\newblock {\em ESAIM Control Optim. Calc. Var.}, 18(3):656--692, 2012.

\bibitem{Rivera_Langenhop_1978}
H. Rivera~Rodas and C.~E. Langenhop.
\newblock A sufficient condition for function space controllability of a linear
  neutral system.
\newblock {\em SIAM J. Control Optim.}, 16(3):429--435, 1978.

\bibitem{Rudin_1975}
W. Rudin.
\newblock {\em Functional analysis}.
\newblock International Series in Pure and Applied Mathematics. McGraw-Hill
  Inc., New York, second edition, 1991.

\bibitem{Shklyar_2011}
B.~Shklyar.
\newblock Exact null controllability of abstract differential equations by
  finite-dimensional control and strongly minimal families of exponentials.
\newblock {\em Diff. Equ. Appl.}, 3(2):171--188, 2011.

\bibitem{Wonham_1985}
W.~M. Wonham, {\em Linear multivariable control: a geometric approach}.
  Springer, New York, 3rd Edition, 1985.


\bibitem{Young_1980}
Robert~M. Young.
\newblock {\em An introduction to nonharmonic {F}ourier series}.
\newblock Academic Press, New York,
  1980.

\end{thebibliography}
\end{document}